\documentclass[12pt]{amsart}
\usepackage{amsmath}
\usepackage{amsthm}
\usepackage{amsfonts}
\usepackage{amssymb}
\newtheorem{Theorem}{Theorem}[section]

\theoremstyle{definition}

\theoremstyle{remark}

\numberwithin{equation}{section}

\newcommand{\R}{{\mathbb R}}
\newcommand{\C}{{\mathbb C}}

\newcommand{\tr}{{\textrm{\rm tr}\:}}

\begin{document}

\title[Oscillation theory]{The essential spectrum of canonical systems}

\author{Christian Remling}

\address{Department of Mathematics\\
University of Oklahoma\\
Norman, OK 73019}
\email{christian.remling@ou.edu}
\urladdr{www.math.ou.edu/$\sim$cremling}

\author{Kyle Scarbrough}

\address{Department of Mathematics\\
University of Oklahoma\\
Norman, OK 73019}

\email{kyle.d.scarbrough-1@ou.edu}
\date{July 29, 2019}



\begin{abstract}
We study the minimum of the essential spectrum of canonical systems $Ju'=-zHu$. Our results can be described as
a generalized and more quantitative version of the characterization of systems with purely discrete spectrum,
which was recently
obtained by Romanov and Woracek \cite{RW}. Our key tool is oscillation theory.
\end{abstract}
\maketitle
\section{Introduction}
A \textit{canonical system }is a differential equation of the form
\begin{equation}
\label{can}
Ju'(x) = -zH(x)u(x) , \quad J=\begin{pmatrix} 0 & -1 \\ 1 & 0 \end{pmatrix} ,
\end{equation}
with a locally integrable coefficient function $H(x)\in\R^{2\times 2}$, $H(x)\ge 0$, $\tr H(x)=1$.
Canonical systems are of fundamental importance in spectral theory because they may be used
to realize arbitrary spectral data \cite[Theorem 5.1]{Rembook}.

We will explicitly discuss only half line problems $x\in [0,\infty)$,
and we always impose the boundary condition
$u_2(0) = 0$
at the (regular) left endpoint $x=0$. The canonical system together with this boundary condition generates
a self-adjoint relation $\mathcal S$ on the Hilbert space
\[
L^2_H(0,\infty) = \left\{ f: (0,\infty)\to\C^2 : \int_0^{\infty} f^*(x)H(x)f(x)\, dx < \infty \right\}
\]
and then also a self-adjoint
operator $S$ on a possibly smaller space, after dividing out the multi-valued part
of $\mathcal S$. We refer the reader to \cite{Rembook} for more on the basic theory.
We are interested in the spectral theory of $S$ and, more specifically, in its essential spectrum,
which we will denote by $\sigma_{ess}(H)$.

We will find it convenient to write our coefficient matrices in the form
\begin{equation}
\label{H}
H(x) = \begin{pmatrix} \sin^2\varphi & g\sin\varphi\cos\varphi \\
g\sin\varphi\cos\varphi & \cos^2\varphi \end{pmatrix} ,
\end{equation}
with $0\le g(x)\le 1$ and, say, $-\pi/2\le\varphi(x)<\pi/2$. Note that for each fixed $x\ge 0$, this is simply an explicit
way of describing the general matrix $H(x)\ge 0$ with $\textrm{tr}\: H(x)=1$.
Conversely, $g(x)$ and the angle $\varphi(x)$, modulo $\pi$, are determined by $H(x)$.

We study the location of the smallest (in absolute value)
point of $\sigma_{ess}$. So let's define
\[
M(H) = \min \{ |t| : t\in\sigma_{ess}(H) \}
\]
As usual in such situations,
we set $M(H)=\infty$ if $\sigma_{ess}=\emptyset$.

We will consider the positive and negative parts of $\sigma_{ess}$ separately,
and, as will become clearer later on, this is an essential feature of
our analysis. It then seems natural to assume that $0\notin\sigma_{ess}$, so that this set indeed splits into
two separate parts. (However, it will later turn out that our results work under a weaker assumption.)

If $0\notin\sigma_{ess}(H)$, then \eqref{can} at $z=0$ has an $L^2_H$ solution \cite[Theorem 3.8(b)]{Rembook}.
Of course, the solutions for $z=0$ are simply constants, so we conclude that
\begin{equation}
\label{1.1}
\int_0^{\infty} v^* H(x) v \, dx < \infty
\end{equation}
for some $v\in \R^2$, $v\not= 0$. We now adopt the set-up from \cite{RW}
and make the more specific assumption that $v=e_1=(1,0)^t$ here so that then \eqref{1.1} becomes the
condition that $\sin\varphi\in L^2(0,\infty)$. The general case can easily be reduced to this situation
by passing to the canonical system $R^*H(x)R$, with $R$ chosen as a rotation matrix
with $Re_1=v$. Recall here that $\sigma_{ess}(R^*HR)=\sigma_{ess}(H)$; this follows, for example,
from \cite[Theorem 3.20]{Rembook}.

One main source of inspiration for this paper is provided by the recent beautiful work
of Romanov-Woracek \cite{RW}, which, in its first part,
characterizes the canonical systems with purely discrete spectrum. In the setting just described, with
the assumption that $\sin\varphi\in L^2$ in place, Romanov-Woracek
prove that $\sigma_{ess}(H)=\emptyset$ if and only if
\begin{equation}
\label{rw}
\lim_{x\to\infty} x \int_x^{\infty} \sin^2\varphi(s)\, ds = 0 .
\end{equation}
The key ingredient to this is a comparison result (which, in \cite{RW}, is called the
independence theorem). This says that the
absence of essential spectrum is not affected by arbitrary changes to the off-diagonal parts of $H(x)$.
In particular, $\sigma_{ess}(H)=\emptyset$ if and only if $\sigma_{ess}(H_d)=\emptyset$, and here
$H_d$ denotes the diagonal canonical system that is obtained by setting the off-diagonal elements
(or, equivalently, $g(x)$) equal to zero in \eqref{H}. This comparison theorem
is of considerable independent interest.
In fact, it is quite surprising since such a removal of the
off-diagonal terms is not a small perturbation in any
traditional sense, and indeed the essential spectra of $H$ and $H_d$ can differ
quite substantially in more general situations.

In a second step, Romanov-Woracek then derive the criterion \eqref{rw} for the diagonal system;
here, similar investigations can be found in the earlier literature, see for example \cite{Kac,KacK,RS1}.

The whole analysis of \cite{RW} proceeds by studying
the resolvent operator $(\mathcal S -z)^{-1}$ for $z=0$, and this method is mainly powered by the fact
that the kernel of this integral operator can be written down explicitly in terms of $H$.

We will give a completely different treatment here. It will be oscillation theoretic, and can thus be viewed
as a continuation of our work begun in \cite{RS1}. Given the significance of the results of \cite{RW},
it might be of some interest to have an alternative approach, and it will again turn out that
oscillation theory is a flexible tool that produces quick and transparent proofs.

More importantly, our new approach will allow us
to establish more quantitative versions of these results
which will apply to more general situations. This we will do for both parts of the analysis of \cite{RW},
the comparison theorem and the actual analysis of $\sigma_{ess}(H_d)$.

The comparison theorem of Romanov-Woracek may be formulated, in our notation, as the statement
that $M(H)=\infty$ if and only if $M(H_d)=\infty$. We will sharpen this as follows here.
\begin{Theorem}
\label{T1.1}
Assume that $\sin\varphi\in L^2(0,\infty)$. Then
\[
\frac{1}{2}\, M(H_d) \le M(H) \le \frac{2}{3-\sqrt{5}} M(H_d) .
\]
\end{Theorem}
The first inequality is optimal, and we will present such an example in the final section.
The second one almost certainly isn't, and in fact we conjecture that the optimal
constant is $1$ rather than our value $2/(3-\sqrt{5})\simeq 2.62$, which does not seem to have any
real significance and is simply delivered by certain computations.
But we have not made any real progress on this question.

The inclusion of both the negative and positive parts of $\sigma_{ess}(H)$ in the definition
of $M(H)$ is crucial here. One-sided versions of Theorem \ref{T1.1} that only deal with the positive (or negative)
part of the essential spectrum fail badly. For example,
one can simply start out with an $H(x)=P_{\varphi(x)}$ that is a projection; in other words,
$g\equiv 1$; see also \eqref{pr} below. Under suitable monotonicity assumptions on $\varphi(x)$
(see \cite{RS1,Win,WW} for more details), such an $H$ will have non-negative spectrum. However,
the spectrum of a diagonal canonical system is always symmetric about zero \cite[Theorem 6.13]{Rembook}.

The fact that specifically the diagonal system $H_d$
serves as the comparison operator is due to our normalization $\sin\varphi\in L^2$.
In general, if we only assume \eqref{1.1}, then, as discussed above, we can use the transformation
$R^*H(x)R$ to return to the situation discussed in Theorem \ref{T1.1}. It is useful to observe here
that the coefficient matrices with a given diagonal (these are the ones that
are being compared in Theorem \ref{T1.1}) can be described as the line segment
$\lambda P_{\varphi} + (1-\lambda)P_{-\varphi}$, with
\begin{equation}
\label{pr}
P_{\beta} = \begin{pmatrix} \sin^2\beta & \sin\beta\cos\beta \\
\sin\varphi\cos\beta & \cos^2\beta
\end{pmatrix}
\end{equation}
denoting the projection onto $(\sin\beta,\cos\beta)^t$, and $\lambda=\lambda(x)$
satisfies $0\le\lambda\le 1$. The diagonal system $H_d$ is obtained by taking $\lambda=1/2$ here.

So in the general case Theorem \ref{T1.1} will be comparing the coefficient functions
$\lambda(x) P_{\alpha+\psi(x)} + (1-\lambda(x)) P_{\alpha-\psi(x)}$, and here $\alpha$ depends
on the $v$ from \eqref{1.1}, and $\psi=\psi(\varphi,g;\alpha)$ depends on the original data
(a straightforward calculation would of course produce an explicit formula).
In particular, note that the role of $H_d$ is now taken over by
$(1/2)(P_{\alpha+\psi}+P_{\alpha-\psi})$, which will usually not be
a diagonal matrix.

We also present a completely different type of comparison result. This will work especially well for
diagonal canonical systems.
\begin{Theorem}
\label{T1.2}
Assume that $\sin\varphi\in L^2(0,\infty)$. Let $L(t)$ be the (self-adjoint) Schr{\"o}dinger operator
on $L^2(0,\infty)$ with, say, Dirichlet boundary conditions $u(0)=0$ that is generated by
\[
L(t) = -\frac{d^2}{dx^2} - t^2\sin^2\varphi(x) ,
\]
and define
\[
S = \sup \{ t\ge 0 : L(t) \textrm{\rm{ has finite negative spectrum}} \} .
\]
Then $M(H_d)=S$.
\end{Theorem}
Note that the potential $-t^2\sin^2\varphi$ is integrable, so the spectrum of $L(t)$ is always purely discrete on
$(-\infty, 0)$ (possibly empty), and the only question about it is whether or not the eigenvalues accumulate at zero.
This alternative does of course not depend on the boundary condition at $x=0$.
Moreover, it follows at once from the min-max principle that we switch from finite to infinite negative spectrum only
once (or not at all), at a unique $t=S$.

It also possible to more generally relate $M(H)$ to the negative spectra
of certain Sturm-Liouville operators (which could in fact
also be rewritten as Schr{\"o}dinger operators), but the results are less complete and satisfying,
and we won't say more on this here. This topic will be the subject of continuing research.
\begin{Theorem}
\label{T1.3}
Assume that $\sin\varphi\in L^2(0,\infty)$, and let
\[
A = \limsup_{x\to\infty} x \int_x^{\infty} \sin^2\varphi(t)\, dt \in [0,\infty] .
\]
Then
\[
\frac{1}{2\sqrt{A}} \le M(H_d) \le \frac{1}{\sqrt{A}} .
\]
\end{Theorem}
Notice that we recover criterion \eqref{rw} as the special case $A=0$ of this.

If the limit exists, then the lower bound becomes accurate. More generally,
we always have the following alternative upper bound.
\begin{Theorem}
\label{T1.4}
Assume that $\sin\varphi\in L^2(0,\infty)$, and let
\[
B = \liminf_{x\to\infty} x \int_x^{\infty} \sin^2\varphi(t)\, dt .
\]
Then
\[
M(H_d)\le \frac{1}{2\sqrt{B}} .
\]
\end{Theorem}

We organize this paper as follows: Section 2 gives a quick review of the key oscillation theoretic
tools that we will need here. We then prove our comparison results, Theorems \ref{T1.1} and \ref{T1.2},
in Section 3 and the bounds on $M(H_d)$ (Theorems \ref{T1.3}, \ref{T1.4}) in Section 4. The final section
will present an example that confirms that the first inequality from Theorem \ref{T1.1} is sharp.
\section{Oscillation theory}
Let us quickly review how $M(H)$ may be found from the oscillatory (or non-oscillatory) behavior of
the solutions to \eqref{can}. This material was developed in detail in \cite{RS1}.

Given a non-trivial solution $u$ of \eqref{can} for $z=t\in\R$, introduce $R(x)>0$, $\theta(x)$ by writing
$u=Re_{\theta}$, with $\theta(x)$ continuous and here $e_{\theta}=(\cos\theta,\sin\theta)^t$.
Then the \textit{Pr{\"u}fer angle }$\theta(x)$ solves
\begin{equation}
\label{ot}
\theta'(x) = te^*_{\theta(x)}H(x)e_{\theta(x)} .
\end{equation}
We see that $\theta$ is monotone in both $t$ and $x$; in particular, $\lim_{x\to\infty} \theta(x)$ always
exists in the generalized sense. We call the equation \eqref{ot}, for a given $t\in\R$,
\textit{oscillatory }if $\lim_{x\to\infty} \theta(x)=\pm\infty$.
Whether or not this holds will not depend on which solution of \eqref{ot} is considered.

If $0\notin\sigma_{ess}$, then we can characterize
\[
M_+(H) := \min \sigma_{ess}(H)\cap (0,\infty)
\]
oscillation theoretically as follows:
\begin{equation}
\label{M+}
M_+(H) = \inf \{ t>0: \eqref{ot} \textrm{ is oscillatory} \}
\end{equation}
Please see \cite[Sections 2, 4]{RS1} for a detailed discussion.

Of course, a similar statement holds for the maximum $M_-(H)$ of the negative part of the essential spectrum
if $0\notin\sigma_{ess}$. Finally, $0\in\sigma_{ess}$ if and only if \eqref{ot} is oscillatory for all $t>0$ or
for all $t<0$ (or both).

The set from the right-hand side of \eqref{M+} always is of the form $(T,\infty)$ or $[T,\infty)$,
for a unique $T\in [0,\infty]$. This is an immediate consequence of the comparison principle for first
order ODEs \cite[Section III.4]{Hart}.
\section{Comparison results}
\begin{proof}[Proof of Theorem \ref{T1.1}]
We start with the first inequality. Of course, there is nothing to prove if
$M(H_d)=0$. Otherwise, let $t\in (0, M(H_d)/2)$. We will then show that
\eqref{ot} is non-oscillatory at this $t$. Since an analogous statement can be
established for $t\in (-M(H_d)/2, 0)$, the desired inequality will then follow from \eqref{M+}.

To do this, we simply write out the right-hand side of \eqref{ot}, which we will denote
by $tf(H)$. So
\begin{equation}
\label{3.1}
tf(H)= t\left( \sin^2\varphi \cos^2\theta +\cos^2\varphi\sin^2\theta +2g\sin\varphi\cos\varphi\sin\theta\cos\theta\right) .
\end{equation}
Obviously, $f(H)\le 2f(H_d)$, and here $tf(H_d)$ similarly denotes the right-hand side of \eqref{ot}
for the diagonal system,
which we can also obtain by simply setting $g=0$ in \eqref{3.1}.

Since the equation for the diagonal system at $2t$ is non-oscillatory, by our choice of $t$,
it now follows from the comparison principle for first order ODEs \cite[Section III.4]{Hart}
that the equation for the full system, at $t$,
has the same property.

The second part of the theorem is proved using similar ideas, but we now need to consider positive and
negative $t$ simultaneously as $M(H_d)$ is, in general, not controlled by just one of $M_{\pm}(H)$.
Let $t\in(0,M(H))$. Then \eqref{ot} is non-oscillatory
at both $t$ and $-t$, and thus the corresponding solutions $\theta (x)$ approach finite limits
when $x\to\infty$. We assumed
that $\sin\varphi\in L^2$, so these limits must be $\equiv 0\bmod \pi$. We may thus fix solutions $\theta_{\pm}$,
corresponding to $t$ and $-t$, respectively, such that $\theta_+(x)\le 0$ and $\theta_-(x)\ge 0$ for all $x\ge a$,
and $\lim_{x\to\infty} \theta_{\pm}(x)=0$. In fact, we will assume that these inequalities are strict. The other case,
when $\theta_{\sigma}(x_0)=0$ at some $x_0\ge a$ for $\sigma=+$ or $\sigma=-$,
is possible only if $\sin\varphi(x)\equiv 0$ for $x\ge x_0$,
and it is easy to show directly that $M(H)=M(H_d)=\infty$ then.

Let's look at $(\theta_+-\theta_-)'$; our plan is to relate this derivative to the right-hand side
of \eqref{ot} for the diagonal system.
We compute
\begin{align*}
\left(\theta_+-\theta_-\right)' & = t \left[ \sin^2\varphi(\cos^2\theta_+ +\cos^2\theta_- )
+\cos^2\varphi (\sin^2\theta_++\sin^2\theta_-) \right. \\
& \quad\quad\left. + g\sin\varphi\cos\varphi (\sin 2\theta_+ +\sin 2\theta_-) \right] \\
& \equiv tF(\theta_+,\theta_-,\varphi) .
\end{align*}
We now claim for any given $\delta>0$, we can estimate
\begin{equation}
\label{3.2}
F \ge (c-\delta) \left( \sin^2\varphi\cos^2(\theta_+-\theta_-) + \cos^2\varphi\sin^2(\theta_+-\theta_-) \right) ,
\end{equation}
provided that $\theta_{\pm}$ are small enough. Here, $c=(3-\sqrt{5})/2$ is the constant from Theorem \ref{T1.1}.
This will imply the desired inequality, by arguing as follows: First of all, the restriction that
$\theta_{\pm}$ are small will not affect the argument because this will always hold at all large $x$,
and of course the equation being oscillatory (or not) is an asymptotic property. Next, observe that indeed
as a function of $\theta\equiv \theta_+-\theta_-$,
the right-hand side of \eqref{3.2} is the right-hand side of \eqref{ot} for the diagonal system,
multiplied by $(c-\delta)/t$.
By our choice of $t$ and the comparison principle, the equation for $H_d$ is thus non-oscillatory
at $(c-\delta)t$. Since $\delta>0$ can be made arbitrarily small here, this shows that $ct\le M(H_d)$, as desired.

So it only remains to establish \eqref{3.2}, and this we do by a series of calculations.
Clearly, the right-hand side of \eqref{3.2}, without the factor $c-\delta$, is bounded from above by
\[
\sin^2\varphi + (\theta_+-\theta_-)^2\cos^2\varphi .
\]
On the other hand,
\[
F \ge (1-\eta) \left[ 2\sin^2\varphi + (\theta_+^2+\theta_-^2)\cos^2\varphi
-2\left| g(\theta_++\theta_-)\sin\varphi\cos\varphi\right| \right] ,
\]
and here $\eta>0$ can be kept arbitrarily small if we make sure that $\theta_{\pm}$ are small enough.
Moreover, we can then absorb $\eta$ by $\delta$ at the end, so it now suffices to show that
\begin{multline*}
2\sin^2\varphi + (\theta_+^2+\theta_-^2)\cos^2\varphi  - 2\left| \theta_++\theta_-\right|\,
\left| \sin\varphi\cos\varphi\right|\ge\\
c \left( \sin^2\varphi +(\theta_+-\theta_-)^2\cos^2\varphi \right) .
\end{multline*}
If we introduce $T=\tan\varphi$ and rearrange slightly, then this becomes
\[
(2-c)T^2 - 2|T|\, \left| \theta_++\theta_-\right| + \theta_+^2+\theta_-^2 - c(\theta_+-\theta_-)^2 \ge 0 .
\]
The minimum of the left-hand side occurs at $|T|=|\theta_++\theta_-|/(2-c)$.
So it is now sufficient to establish that
\[
-\frac{(\theta_++\theta_-)^2}{2-c} + \theta_+^2+\theta_-^2 - c(\theta_+-\theta_-)^2 \ge 0 .
\]
We rewrite this one more time by introducing $q=-\theta_+/\theta_-$. Then what we want to show becomes
\[
-\frac{(-q+1)^2}{2-c} + q^2 + 1 - c(q+1)^2 \ge 0
\]
or, equivalently,
\[
(c^2-3c+1)q^2 + 2(c-1)^2 q + c^2-3c+1 \ge 0 .
\]
Since $c^2-3c+1=0$ for our choice of $c=(3-\sqrt{5})/2$ and $q>0$, this last inequality is obvious.
\end{proof}
\begin{proof}[Proof of Theorem \ref{T1.2}]
To relate our (diagonal) canonical systems to the Schr{\"o}dinger operator $L(t)$, we exploit the well known
fact that Schr{\"o}dinger equations may be rewritten as Riccati equations, and this will get us to \eqref{ot}.
More precisely, \eqref{ot} is not itself a Riccati equation, but will become one after a suitable Taylor expansion
of the right-hand side.

In the first part of the argument, we show that $M(H_d)\le S$.
Let $t\in (0,M(H_d))$. So \eqref{ot}, for the diagonal system, has a solution $\theta$ with $\theta(x)\le 0$
for all $x\ge a$. By taking $a>0$ large enough here, we may also assume that $\theta(x)$ is small.
We then estimate the right-hand side of \eqref{ot} as follows:
\begin{align*}
\cos^2\varphi\sin^2\theta + \sin^2\varphi\cos^2\theta & = \sin^2\theta + \cos 2\theta \sin^2\varphi\\
& \ge (1-\eta)(\theta^2 + \sin^2\varphi)
\end{align*}
Here, $\eta>0$ can be made arbitrarily small by making sure that $\theta(x)$ is small. In particular,
we can demand that also $t<(1-\eta)M(H_d)$, and then the comparison principle shows that the solution $\theta_1$ of
\[
\theta'_1 = t(\theta_1^2 +\sin^2\varphi), \quad\quad \theta_1(a)=\theta(a)
\]
will also satisfy $\theta_1(x)\le 0$, $x\ge a$. Now, as announced, we rewrite this Riccati equation as a
Schr{\"o}dinger equation, by using the transformation $\theta_1=-u'/(tu)$. More precisely, we define
\[
u(x) = \exp \left( -t\int_a^x \theta_1(s)\, ds \right) .
\]
A quick calculation then shows that this $u$ solves $-u''-t^2(\sin^2\varphi) u=0$, and obviously $u(x)>0$
for $x\ge a$. Thus classical oscillation theory for this Schr{\"o}dinger equation (see, for example,
\cite{WMLN}) shows that its negative spectrum is finite. It follows that $t\le S$, so $M(H_d)\le S$, as claimed.

Conversely, let now $t\in (0,S)$. Classical oscillation theory then implies that the solutions $u$ of
\begin{equation}
\label{se}
-u''(x)-t^2(\sin^2\varphi(x))u(x)=0
\end{equation}
have only finitely many zeros on $x>0$. This is not exactly what we need here. Rather,
we would like to have a solution whose \textit{derivative }is zero free eventually. In general,
that is not an equivalent condition, but here we can establish it anyway, thanks to the negativity
of the potential $-t^2\sin^2\varphi$.

We proceed as follows:
First of all, fix an $a>0$
such that there is a solution $u(x)$ of \eqref{se}
that is zero free on $[a,\infty)$. Then, again by classical oscillation theory,
an arbitrary solution has at most one zero on this interval. Consider now specifically the solution $u$ with
$u(a+1)=u'(a+1)=1$. If we had $u>0$ throughout the interval $[a,a+1]$, then it would follow that
$u''=-t^2(\sin^2\varphi)u\le 0$ there, and this concavity together with the initial conditions at $a+1$ would
give $u$ a zero on $[a,a+1]$ after all. So we have to admit that there is such a zero,
and hence $u>0$ on $[a+1,\infty)$.
Again, this shows us that $u''\le 0$ there. So if we had $u'(x_0)=0$ at some $x_0>a+1$, then, by the concavity of $u$,
a second zero can only be avoided if $\sin^2\varphi\equiv 0$ on $[x_0,\infty)$. This is a trivial scenario since
then clearly $u=x$ is a solution of the desired type (in fact, it can again easily be shown directly that
$M(H_d)=S=\infty$ in this case). So, to summarize, we have found a solution $u$ with $u(x), u'(x)>0$
for $x\ge a+1$.

We now estimate the right-hand side of \eqref{ot} from above by $t(\theta^2+\sin^2\varphi)$.
Then we again
consider the corresponding comparison equation
\begin{equation}
\label{3.3}
\theta'_2 = t(\theta_2^2 +\sin^2\varphi) .
\end{equation}
The same transformation as above shows us that this is solved, on $[a+1,\infty)$, by $\theta_2=-u'/(tu)$,
and here we of course use the solution $u$ that was constructed above. It follows that $\theta_2(x)<0$ on $x\ge a+1$.
The comparison principle thus implies that the original equation \eqref{ot}, for $H_d$, is non-oscillatory.
Hence $t\le M(H_d)$, as required; recall here again that the spectrum of a \textit{diagonal }canonical system
is symmetric about zero \cite[Theorem 6.13]{Rembook}, so there is no need to consider
negative values of $t$ separately.
\end{proof}
\section{Bounds on $M(H_d)$}
We present proofs of Theorems \ref{T1.3}, \ref{T1.4} here. These results are also closely related
to the bounds for canonical systems with non-negative spectrum that
we established in \cite{RS1}; see Theorems 4.1, 4.2 there. To make this connection explicit,
one can use the transformations between diagonal and semibounded canonical systems that were discussed
in detail in \cite[Section 5]{RS1}. Somewhat surprisingly perhaps, the bounds one obtains in this way
are similar but not identical to the ones we prove here.
\begin{proof}[Proof of Theorem \ref{T1.3}]
To establish the first inequality, we again work with \eqref{3.3} as our comparison equation.
Let $t>0$. Introduce
\[
W(x) = \int_x^{\infty} \sin^2\varphi(s)\, ds
\]
and then $\alpha =t\theta_2+t^2W$. A quick calculation shows that this solves
\begin{equation}
\label{4.6}
\alpha'=(\alpha-t^2W)^2 .
\end{equation}
Moreover, for any $B>A$ and all sufficiently large $x$, we may estimate the right-hand side of this equation
from above by $(\alpha(x)-t^2B/x)^2$, at least as long as $\alpha(x)\le 0$.

So let us now look at the corresponding comparison equation
\begin{equation}
\label{4.1}
\alpha'_1 = \left( \alpha_1 - \frac{t^2B}{x} \right)^2 .
\end{equation}
We established in \cite{RS1} that if $t^2B<1/4$, then \eqref{4.1} will have a global solution
$\alpha_1(x)\le 0$ on a suitable half line $x\ge a$.
Please see eqn.\ (4.3) of \cite{RS1} and the discussion that follows for the details.

So we can now retrace our steps and apply the comparison principle repeatedly. Note that along the way,
we will confirm that $\alpha(x)\le\alpha_1(x)\le 0$, and this justifies the step from \eqref{4.6}
to \eqref{4.1}. We conclude that \eqref{ot}
is non-oscillatory. The upshot of all this is the statement that if $t^2B<1/4$,
then $t\le M(H_d)$. Since $B>A$ can be arbitrarily close
to $A$ here, this implies that $M(H_d)\ge 1/(2\sqrt{A})$, as asserted.

The second inequality can be proved in similar style, by discussing \eqref{ot} and suitable comparison equations
directly. An alternative, very transparent argument can be given with the help of Theorem \ref{T1.2},
and this we will now present.

Let $t\in (0,S)$. Then for all large $a>0$,
the Schr{\"o}dinger operator $L=-d^2/dx^2-t^2\sin^2\varphi(x)$ has \textit{empty }negative
spectrum when considered on $L^2(a,\infty)$ with Dirichlet boundary conditions $u(a)=0$ at $x=a$.
For example, this follows from classical oscillation theory.

Thus the quadratic form
\[
Q(u) = \int_a^{\infty} \left( u'(x)^2 - t^2(\sin^2\varphi(x)) u^2(x) \right)\, dx
\]
will satisfy $Q(u)\ge 0$ for all test functions $u\in H^1$, $u(a)=0$. Consider now a tent shaped test function
\[
u(x) = \begin{cases} \frac{x-a}{b-a} & x<b \\ 1 & b<x<c \\ \frac{d-x}{d-c} & c<x<d \\ 0 & x>d \end{cases} ;
\]
the first three intervals will eventually be taken very large.

By the definition of $A$, for any $B<A$, we can find arbitrarily large $b>a$ such that
$\int_b^{\infty}\sin^2\varphi\, dx \ge B/b$. Fix such a $b$ for the moment. Then
\[
Q(u) \le \frac{1}{b-a} - \frac{t^2B}{b} +o(1) ,
\]
where the error term denotes a contribution that can be made arbitrarily small by sending $c,d-c\to\infty$.
Since, as remarked, $Q(u)\ge 0$ and $b$ can be arbitrarily large, it follows that $t^2B\le 1$.
Since $t$ and $B$ can be arbitrarily close to $S$ and $A$, respectively, this says that $S\le 1/\sqrt{A}$, as claimed.
\end{proof}
\begin{proof}[Proof of Theorem \ref{T1.4}]
We again use Theorem \ref{T1.2}. Let $t\in (0,S)$. As in the proof of Theorem \ref{T1.2},
we again work with a solution $u$ of \eqref{se} with initial values $u(a)=u'(a)=1$. We saw above
that then $u(x)\not=0$ for $x\ge a$ if we take $a>0$ large enough here.
So we can now introduce $\alpha=u'/u-t^2W$, with
\[
W(x)= \int_x^{\infty}\sin^2\varphi(s)\, ds ,
\]
again as above. Then $\alpha$ solves
\begin{equation}
\label{4.4}
\alpha' = - \left( \alpha+t^2W \right)^2 .
\end{equation}
For any $C<B$ and for all sufficiently large $x$, we then have the upper bound $-(\alpha+t^2C/x)^2$
on the right-hand side, provided that $\alpha(x)\ge 0$. This will hold at least initially, for $x\ge a$
close to $a$, if we took $a$ large enough so that $W(x)$ is already small there. However, then we
observe that in fact $\alpha(x)\ge 0$ for all $x\ge a$: if we had $\alpha(b)=-\alpha_0<0$, then obviously
also $\alpha(x)\le -\alpha_0$ for all $x\ge b$. Moreover, $W(x)\to 0$, so comparison
of \eqref{4.4} with (say)
\[
\alpha'=-\frac{1}{2}\,\alpha^2, \quad\quad \alpha(c)=-\alpha_0
\]
for a suitable (large) $c\ge b$
would show that $\alpha(x)\to-\infty$ in finite time, when we sent $x\to d$ for some $d>0$. This blatant contradiction
to our earlier definition of $\alpha$ as a globally defined function shows that $\alpha(x)\ge 0$ for all
$x\ge a$. In particular, our upper bound on the right-hand side of \eqref{4.4} is valid for all these $x$.

So consider now the corresponding comparison initial value problem
\begin{equation}
\label{4.11}
\alpha'_1 = - \left( \alpha_1 + \frac{t^2C}{x} \right)^2, \quad\quad \alpha_1(b)=\alpha(b) .
\end{equation}
We know that this has a global solution $\alpha_1(x)\ge 0$, $x\ge b$. Moreover, \eqref{4.11}
is essentially the same equation that we already encountered above, as \eqref{4.1}.
We can solve explicitly by defining $\beta=\alpha_1+t^2C/x$. Then
\begin{equation}
\label{4.3}
\beta' = -\beta^2 - \frac{t^2C}{x^2} ,
\end{equation}
and this Riccati equation we can rewrite as a Schr{\"o}dinger equation
\begin{equation}
\label{4.2}
-y'' - \frac{t^2C}{x^2}\, y = 0
\end{equation}
by writing $\beta=y'/y$, or, more precisely, by defining
\[
y(x)=\exp\left( \int_b^x\beta(s)\, ds\right) .
\]
It is well known (and easy to show, by solving it explicitly as an Euler equation)
that \eqref{4.2} will be oscillatory if $t^2C>1/4$, so all solutions $y$ have infinitely many zeros in this case.
We already have a solution $y$ that is obviously zero free, so we have to admit that $t^2C\le 1/4$.
It follows that $S\le 1/(2\sqrt{B})$, as claimed.
\end{proof}
\section{An example}
Let's start out with a Schr{\"o}dinger equation
\begin{equation}
\label{se1}
-y'' + \frac{1}{4}\, y = zy
\end{equation}
with constant potential $V(x)\equiv 1/4$. (Any non-zero constant value would work here, but this choice
will slightly simplify the calculation by saving us one additional step that would otherwise be necessary.)
As is well known \cite[Section 1.3]{Rembook}, we can rewrite \eqref{se1} as an equivalent canonical system.
To do this, we can introduce $Y=(y',y)^t$, so that then $y$ solves \eqref{se1} if and only if $Y$ solves
\begin{equation}
\label{5.1}
Y' = \begin{pmatrix} 0 & 1/4-z \\ 1 & 0 \end{pmatrix} Y .
\end{equation}
Let
\[
T_0 = \begin{pmatrix} p' & q' \\ p & q
\end{pmatrix}
\]
be the matrix solution of \eqref{5.1} for $z=0$ that is built from the solutions $p=e^{x/2}$, $q=e^{-x/2}$ of \eqref{se1}.
Notice that these are chosen such that $\det T_0(x)=1$.

Now introduce $u(x)$ by writing $Y=T_0u$. Then a straightforward calculation shows that $Y$ solves \eqref{5.1}
if and only if $u$ solves the canonical system \eqref{can} with coefficient function
\[
H(x) = \begin{pmatrix}
p^2 & pq \\ pq & q^2 \end{pmatrix} = \begin{pmatrix}
e^x & 1 \\ 1 & e^{-x}
\end{pmatrix}
\]
(which is not trace normed, but this won't matter here).
The corresponding diagonal system
\[
H_d(x) = \begin{pmatrix}
e^x & 0 \\ 0 & e^{-x}
\end{pmatrix}
\]
satisfies $\det H_d(x)=1$, and such diagonal systems can be rewritten as (special) Dirac equations
\[
Jv' = \begin{pmatrix} 0 & W(x) \\ W(x) & 0 \end{pmatrix}v - zv ,
\]
by using a transformation quite similar to the one we just employed. This is discussed in more detail
in \cite[Section 6.4]{Rembook}, and it is also shown there that if $H_d=\textrm{diag}\: (a,a^{-1})$,
then $W=a'/(2a)$. So $W\equiv 1/2$ here.

Of course, it is a trivial matter to find the bottom of the essential spectrum for these Schr{\"o}dinger
and Dirac operators with constant potentials. In this way, we see that $M(H)=1/4$, $M(H_d)=1/2$. (Alternatively,
$M(H_d)$ can also be found quite conveniently with the help of Theorems \ref{T1.3}, \ref{T1.4}, without
going through the Dirac system.) This shows
that the first inequality of Theorem \ref{T1.1} is sharp.

\end{document}